\documentclass[leqno,11pt]{amsart}

\usepackage{verbatim}
\usepackage{amsmath,amscd,amsthm,amsxtra,amssymb}
\usepackage{epsfig,graphics,color,colortbl}
\usepackage{amssymb,latexsym}
\usepackage{mathrsfs}
\usepackage[poly,all]{xy}
\usepackage{marginnote}
\usepackage{oldgerm}
\usepackage[colorlinks=true, pdfstartview=FitV, linkcolor=blue,citecolor=blue,urlcolor=blue]{hyperref}

\usepackage[usenames,dvipsnames,svgnames,table]{xcolor}

\newtheorem{thm}{Theorem}[section]
\newtheorem{prop}[thm]{Proposition}
\newtheorem{cor}[thm]{Corollary}
\newtheorem{lem}[thm]{Lemma}

\theoremstyle{definition}
\newtheorem{defn}[thm]{Definition}

\theoremstyle{remark}
\newtheorem{Rmk}[thm]{Remark}

\numberwithin{equation}{section}

\newcommand{\Z}{\mathbb{Z}}
\newcommand{\Q}{\mathbb{Q}}

\newcommand{\A}{\mathbb{A}}

\newcommand{\g}{\mathfrak{g}}

\newcommand{\Hom}{\mathrm{Hom}}

\newcommand{\Ht}{{\rm ht}}

\newcommand{\Ind}{{\rm Ind}}
\newcommand{\Res}{{\rm Res}}

\newcommand{\gmod}{{\rm \text{-}gmod}}
\newcommand{\proj}{\text{-}\mathrm{proj}}
\newcommand{\Mod}{\text{-}\mathrm{Mod}}
\newcommand{\im}{\mathrm{Im}}

\newcommand{\rlQ}{\mathsf{Q}}   
\newcommand{\wlP}{\mathsf{P}}   
\newcommand{\weyl}{\mathsf{W}}  
\newcommand{\cmA}{\mathsf{A}}  
\newcommand{\tf}{\tilde{f}}  
\newcommand{\te}{\tilde{e}}  

\newcommand{\wt}{\mathrm{wt}} 		
\newcommand{\ep}{\varepsilon}  		
\newcommand{\ph}{\varphi}  		
 
\newcommand\Aqq[1][{\mathfrak g}]{U_{\A}^{-}(#1)^\vee}
\newcommand{\rD}{\Delta}  		

\newcommand{\ST}{\mathsf{ST}}   
\newcommand{\res}{\mathsf{res}}   
\newcommand{\sg}{\mathfrak{S}}   

\newcommand{\bR}{\mathbf{k}}   

\newcommand{\conv}{{\mathbin{\scalebox{1.1}{$\mspace{1.5mu}\circ\mspace{1.5mu}$}}}}

\newcommand{\hd}{{\mathrm{hd}}}      					 

\newcommand{\lan}{\langle} 	
\newcommand{\ran}{\rangle}	

\newcommand{\Par}{\mathfrak{P}}					
\newcommand{\Sp}{\mathcal{S}}					

\newcommand{\Dy}{\mathcal{D}}
\newcommand{\Ix}{\mathsf{I}}
\newcommand{\La}{\Lambda}
\newcommand{\la}{\lambda}
\newcommand{\cB}{\mathbb{B}}
\newcommand{\Pa}{\mathcal{P}}
\newcommand{\Cb}{\mathcal{C}}
\newcommand{\bp}{\mathsf{p}}
\newcommand{\id}{\mathrm{id}}

\newcommand{\Bs}[1] 
{
\xy
(-2.5,2.5)*{}; (2.5,2.5)*{} **\dir{-};
(-2.5,-2.5)*{}; (2.5,-2.5)*{} **\dir{-};
(-2.5,2.5)*{}; (-2.5,-2.5)*{} **\dir{-};
(2.5,2.5)*{}; (2.5,-2.5)*{} **\dir{-};
(0,0)*{#1};
\endxy
}

\newcommand{\Bl}[1] 
{
\xy
(-5,2.5)*{}; (5,2.5)*{} **\dir{-};
(-5,-2.5)*{}; (5,-2.5)*{} **\dir{-};
(-5,2.5)*{}; (-5,-2.5)*{} **\dir{-};
(5,2.5)*{}; (5,-2.5)*{} **\dir{-};
(0,0)*{#1};
\endxy
}

\newcommand{\CrystalC}
{
\xymatrix{
 \Bs{\overline{n}}   & \ar[l]_n  \cdots  & \ar[l]_{3}
\Bs{\overline{2}} & \ar[l]_{ 2} \Bs{\overline{1}}  & \ar[l]_1
\Bs{1} & \ar[l]_{2} \Bs{2} & \ar[l]_{ 3}
 \cdots & \ar[l]_{n} \Bs{n}
}
}

\newcommand{\DynkinA}
{
\xy
(0.5,0)*{};(7.5,0)*{} **\dir{-};
(8.5,0)*{};(12.5,0)*{} **\dir{-};
(12.5,0)*{};(24.5,0)*{} **\dir{.};
(24.5,0)*{};(28.5,0)*{} **\dir{-};
(29.5,0)*{}; (36.5,0)*{} **\dir{-};
(0,0)*{\circ}; (8,0)*{\circ};
(29,0)*{\circ}; (37,0)*{\circ};
(0,-3)*{1}; (8,-3)*{2};
(29,-3)*{n-1};(37,-3)*{n};
(-8,-1)*{A_n\ : \ }
\endxy
}

\newcommand{\DynkinB}
{
\xy
{\ar@{<=} (0.5,0)*{}; (7.5,0)*{}};
(8.5,0)*{};(12.5,0)*{} **\dir{-};
(12.5,0)*{};(24.5,0)*{} **\dir{.};
(24.5,0)*{};(28.5,0)*{} **\dir{-};
(29.5,0)*{};(36.5,0)*{} **\dir{-};
(0,0)*{\circ}; (8,0)*{\bullet};
(29,0)*{\bullet}; (37,0)*{\bullet};
(0,-3)*{1}; (8,-3)*{2};
(29,-3)*{n-1};(37,-3)*{n};
(-8,-1)*{B_n\ : \ }
\endxy
}

\newcommand{\DynkinC}
{
\xy
{\ar@{=>} (0.5,0)*{}; (7.5,0)*{}};
(8.5,0)*{};(12.5,0)*{} **\dir{-};
(12.5,0)*{};(24.5,0)*{} **\dir{.};
(24.5,0)*{};(28.5,0)*{} **\dir{-};
(29.5,0)*{};(36.5,0)*{} **\dir{-};
(0,0)*{\bullet}; (8,0)*{\bullet};
(29,0)*{\bullet}; (37,0)*{\circ};
(0,-3)*{1}; (8,-3)*{2};
(29,-3)*{n-1};(37,-3)*{n};
(-8,-1)*{C_n\ : \ }
\endxy
}

\newcommand{\DynkinD}
{
\xy
(0,2.5)*{}; (5.1,0.7)*{} **\dir{-};
(0,-2.5)*{}; (5.1,-0.5)*{} **\dir{-};
(5.5,0)*{};(12.5,0)*{} **\dir{-};
(13.5,0)*{};(20.5,0)*{} **\dir{-};
(20.5,0)*{};(27.5,0)*{} **\dir{.};
(27.5,0)*{};(34.5,0)*{} **\dir{-};
(5,0)*{\bullet}; (13,0)*{\bullet}; (28,0)*{\bullet};
(35,0)*{\circ}; (-0.5,2.5)*{\circ}; (-0.5,-2.8)*{\circ};
(5,-3)*{3}; (13,-3)*{4};
(35,-3)*{n};(-1.3,0.5)*{1}; (-1.3,-5)*{2};
(-8,-1)*{D_n\ : \ }
\endxy
}

\newcommand{\DynkinEs}
{
\xy
(0.5,0)*{};(7.5,0)*{} **\dir{-};
(8.5,0)*{};(15.5,0)*{} **\dir{-};
(16.5,0)*{};(23.5,0)*{} **\dir{-};
(24.5,0)*{};(31.5,0)*{} **\dir{-};
(16,0.5)*{};(16,5.5)*{} **\dir{-};
(0,0)*{\circ}; (8,0)*{\bullet}; (16,0)*{\bullet};
(24,0)*{\bullet}; (32,0)*{\circ};
(16,6)*{\bullet};
(0,-3)*{1}; (8,-3)*{2}; (16,-3)*{3};
(24,-3)*{5}; (32,-3)*{6};
(14,5)*{4};
(-8,-1)*{E_6\ : \ }
\endxy
}

\newcommand{\DynkinEsv}
{
\xy
(0.5,0)*{};(7.5,0)*{} **\dir{-};
(8.5,0)*{};(15.5,0)*{} **\dir{-};
(16.5,0)*{};(23.5,0)*{} **\dir{-};
(24.5,0)*{};(31.5,0)*{} **\dir{-};
(32.5,0)*{};(39.5,0)*{} **\dir{-};
(16,0.5)*{};(16,5.5)*{} **\dir{-};
(0,0)*{\bullet}; (8,0)*{\bullet}; (16,0)*{\bullet};
(24,0)*{\bullet}; (32,0)*{\bullet};
(40,0)*{\circ};
(16,6)*{\bullet};
(0,-3)*{1}; (8,-3)*{2}; (16,-3)*{3};
(24,-3)*{5}; (32,-3)*{6}; (40,-3)*{7};
(14,5)*{4};
(-8,-1)*{E_7\ : \ }
\endxy
}

\newcommand{\DynkinF}
{
\xy
(0.5,0)*{};(7.5,0)*{} **\dir{-};
{\ar@{=>} (8.5,0)*{}; (15.5,0)*{}};
(16.5,0)*{};(23.5,0)*{} **\dir{-};
(0,0)*{\circ}; (8,0)*{\circ};
(16,0)*{\circ}; (24,0)*{\circ};
(0,-3)*{1}; (8,-3)*{2};
(16,-3)*{3};(24,-3)*{4};
\endxy
}

\begin{document}

\title[Cyclotomic quiver Hecke algebras and minuscule representations]
{Cyclotomic quiver Hecke algebras corresponding to minuscule representations}

\author[Euiyong Park]{Euiyong Park}
\thanks{The research was supported by the National Research Foundation of Korea (NRF) Grant funded by the Korean Government (NRF-2017R1A1A1A05001058).}
\address{Department of Mathematics, University of Seoul, Seoul 02504, Korea}
\email{epark@uos.ac.kr}


\begin{abstract}
In the paper, we give an explicit basis of the cyclotomic quiver Hecke algebra corresponding to a minuscule representation of finite type.  
\end{abstract}

\maketitle


\vskip 2em

\section*{Introduction}

The quiver Hecke algebras (also known as Khovanov-Lauda-Rouquier algebras) were introduced by Khovanov and Lauda \cite{KL09, KL11} and independently Rouquier \cite{R08} to categorfy the half of a quantum group.
For a dominant integral weight $\La$, the quiver Hecke algebras have special quotient algebras $R^\La$ which give  a categorification of the irreducible highest weight module $V(\La)$ \cite{KK11}.
In the viewpoint of categorification, these algebras have been the focus of many studies and various new features were discovered.

In the paper, we give an explicit basis of the cyclotomic quiver Hecke algebra corresponding to a \emph{minuscule representation} in terms of crystals.  
A minuscule representation is an irreducible highest weight $U_q(\g)$-module $V(\La_i)$ such that the Weyl group acts transitively on its weight space. 
The possible types for minuscule representations are $A_n$, $B_n$, $C_n$, $D_n$, $E_6$ and $E_7$, and the nontrivial highest weights of minuscule representations are marked by $\circ$ in $\eqref{Eq: Dynkin}$. 
Let $V(\La_i)$ be a minuscule representation. Since every weight space $V(\La_i)_\xi$ is extremal, the corresponding cyclotomic quiver Hecke algebra $R^{\La_i}(\xi)$ is simple.
Thus, since every irreducible $R$-module is absolutely irreducible,
 we obtain an explicit basis of $R^{\La_i}(\xi)$ by constructing a simple $R^{\La_i}(\xi)$-module combinatorially in terms of the crystal $B(\La_i)$.
Our main tool is a \emph{homogeneous representation}, which was introduced by Kleshchev and Ram \cite{KR10}. In \cite{KR10},  they gave a combinatorial construction of homogeneous representations of simply-laced type and investigated a connection with \emph{fully commutative elements} including dominant minuscule elements. 
The homogeneous representations of type $A_n$ and $D_n$ are classified and enumerated in \cite{FL18A, FL18B} 
by studying fully commutative elements.  

For simply-laced $ADE$ type, essential parts for the main theorem come from the results of \cite{KR10}.
Let $b \in B(\La_i)$,  $\xi =  \wt(b)$ and $\Pa(b)$ the set of all paths from $b_{\La_i}$ to $b$ on the crystal $B(\La_i)$ (see $\eqref{Eq: paths}$).
We show that $\Pa(b)$ is in 1-1 correspondence to the set of reduced expressions of the minuscule element $w(b)$, 
which says that $\Pa(b)$ has a homogeneous $R^{\La_i}(\xi)$-module structure by the result of  \cite{KR10} (see Theorem \ref{Thm: main ADE}).  
For type $C_n$, since a simple $R^{\La_n}(\xi)$-module is 1-dimensional for all weights $\xi$, we easily obtain that the algebra $R^{\La_n}(\xi)$ is trivial (see Proposition \ref{Prop: main C}).
The main novelty of this paper lies in the case of type $B_n$. Note that, since $B_n$ is not simply-laced, we cannot use the result of \cite{KR10} directly. 
In type $B_n$, we show that there is a 1-1 correspondence between 
the set $\Par_n$ of all strict partitions $\la = (\la_1, \la_2, \ldots )$ with $\la_1 \le n$ and the Weyl group orbit $\weyl\cdot \La_1$. We then 
give a homogeneous $R^{\La_1}(\xi)$-module structure on the set $\ST(\la)$ of standard tableaux of shape $\la$ using combinatorics of shifted Young tableaux. 
Thus, we obtain a $\bR$-basis of $R^{\La_1}(\xi)$ indexed by the set of all pairs of shifted Young tableaux (see Theorem \ref{Thm: main B}).
Remark \ref{Rmk: ST} tells us that shifted Young tableaux also can be regarded as paths on the crystal $B(\La_1)$. As a corollary,
we show that there exist algebra isomorphisms
$R^{\La_1}_{B_n}(m) \simeq R^{\La_1}_{D_{n+1}}(m) \simeq R^{\La_2}_{D_{n+1}}(m)$, where $R^{\La_i}_X(m)$ is the cyclotomic quiver Hecke algebra of type $X$  (see Corollary \ref{Cor: isom}).

This paper is organized as follows:
In Section \ref{Sec: minuscule repn}, we review briefly minuscule representations of quantum groups of finite types.
In Section \ref{Sec: CQHA for ACDE}, we review cyclotomic quiver Hecke algebras and show an explicit basis of $R^{\La_i}(\xi)$ corresponding to minuscule representations for finite types $ADE$ and $C$.
In Section \ref{Sec: Irreducible repn}, we construct simple $R^{\La_1}$-module of type $B_n$ using combinatorics of shifted Young tableaux and show an explicit basis of $R^{\La_1}(\xi)$.

\vskip 2em
\noindent
{\bf Acknowledgments.}
The author would like to thank Myungho Kim for fruitful discussions.
The author also would like to thank the anonymous reviewers for their valuable comments and suggestions.

\section{Minuscule representations} \label{Sec: minuscule repn}

\subsection{Quantum groups}
Let $I = \{ 1, 2, \ldots, n \}$. A quintuple $ (\cmA,\wlP,\Pi,\wlP^\vee,\Pi^\vee) $ is called a {\it Cartan datum} if it
consists of
\begin{enumerate}
\item[(a)] a generalized \emph{Cartan matrix} $\cmA=(a_{ij})_{i,j\in I}$, 
\item[(b)] a free abelian group $\wlP$, called the {\em weight lattice},
\item[(c)] $\Pi = \{ \alpha_i \mid i\in I \} \subset \wlP$,
called the set of {\em simple roots},
\item[(d)] $\wlP^{\vee}=
\Hom_{\Z}( \wlP, \Z )$, called the \emph{coweight lattice},
\item[(e)] $\Pi^{\vee} =\{ h_i \in \wlP^\vee \mid i\in I\}$, called the set of {\em simple coroots} 
\end{enumerate}
satisfying the following:
\begin{enumerate}
\item[(i)] $\lan h_i, \alpha_j \ran = a_{ij}$ for $i,j \in I$,
\item[(ii)] $\Pi$ is linearly independent over $\Q$,
\item[(iii)] for each $i\in I$, there exists $\Lambda_i \in \wlP$, called the \emph{fundamental weight}, such that $\lan h_j,\Lambda_i \ran =\delta_{j,i}$ for all $j \in I$.
\end{enumerate}

There is a nondegenerate symmetric bilinear form $( \cdot \, , \cdot )$ on $\wlP$ satisfying 
\begin{equation*}
(\alpha_i,\alpha_j)=\mathsf{d_i} a_{ij}\qquad \text{and} \qquad  \lan h_i,  \lambda\ran = \dfrac{2 (\alpha_i,\lambda)}{(\alpha_i,\alpha_i)},
\end{equation*}
where $\mathsf d_i$'s are relatively prime integers such that ${\rm diag}(\mathsf d_i \mid i \in I) \cdot \cmA$ is symmetric.
Set $ \rlQ := \bigoplus_{i \in I} \Z \alpha_i$ and  $\rlQ_+ := \sum_{i\in I} \Z_{\ge 0} \alpha_i$ and define $\Ht (\beta)=\sum_{i \in I} k_i $  for $\beta=\sum_{i \in I} k_i \alpha_i \in \rlQ_+$.
We denote by $\wlP_+: =\{ \lambda \in \wlP \mid \lan h_i, \lambda\ran \ge 0 \ \text{for all }  \ i \in I \}$ the set of \emph{dominant integral weights}. 

Let $U_q(\g)$ be the \emph{quantum group} associated with the Cartan datum $(\cmA, \wlP,\wlP^\vee, \Pi, \Pi^{\vee})$, which
is a $\Q(q)$-algebra generated by $f_i$, $e_i$ $(i\in I)$ and $q^h$ $(h\in \wlP^\vee)$ with certain defining relations (see \cite[Chater 3]{HK02} for details).
For $\lambda \in \wlP_+$, we set $V(\lambda)$ to be the irreducible highest weight $U_q(\g)$-module with highest weight $\lambda$ and denote by $B(\lambda)$ its crystal 
(see \cite{Kas90,Kas91}, \cite[Chapter 4]{HK02} for the definition). 
Let $\A = \Z[q,q^{-1}]$. We denote by $U^-_\A(\g)$ and $V_\A(\la)$ the $\A$-lattices of $U_q^-(\g)$ and $V_q(\la)$ respectively.

\subsection{Minuscule representations} \label{Sec: minuscule rep}

We write $\rD$ for the set of all \emph{roots} associated with $\cmA$, and denote by $\weyl$  the \emph{Weyl group}, which is the subgroup of $\mathrm{Aut}(\wlP)$ generated by  
$ s_i(\lambda) := \lambda - \langle h_i, \lambda \rangle \alpha_i  $ for $i\in I$. 
When $w\in \weyl$ is written as $ w = s_{i_1} \ldots s_{i_\ell}$  ($\ell$ minimal), we call the expression \emph{reduced} and write $\ell(w) = \ell$.
An element $w$ is \emph{fully commutative} if for every pair of noncommuting generators $s_i$ and $ s_j$ there is no reduced expression for $w$ containing a subword of length $m$ of the form $s_is_js_is_j \ldots$, where $m$ denotes the order of $s_is_j$ in $\weyl$.

For $\la\in \wlP_+$, an element $w \in \weyl$ is $\la$-\emph{minuscule} if there exists a reduced expression $w = s_{i_1} \cdots s_{i_\ell} $ such that 
\begin{align} \label{Eq: def of minuscule}
\langle h_{i_k},  s_{i_{k+1}} s_{i_{k+2}} \cdots s_{i_{\ell}} \la   \rangle = 1 \qquad \text{ for } k=1,\ldots, \ell.
\end{align}
A weight $\lambda \in \wlP$ is called \emph{minuscule} if $\langle \alpha^\vee, \lambda \rangle \in \{ 0,\pm1\}$ for all $\alpha\in \rD$, where $\alpha^\vee$ is the coroot corresponding to $\alpha$. 
Irreducible highest weight modules corresponding to the dominant minuscule weights are called \emph{minuscule representations}.
In the following Dynkin diagrams, the dominant minuscule weights $\Lambda_i$ are marked by $\circ$ (See \cite[Chapter VIII, $\S$ 7.3]{Bourbaki789} for details). 
\begin{equation} \label{Eq: Dynkin}
\begin{aligned}
\begin{tabular}{ l l l  }
   \DynkinA & & \DynkinB  \\
  \DynkinC & & \DynkinD \\
  \DynkinEs  & & \DynkinEsv  \\
\end{tabular}
\end{aligned}
\end{equation}
\vskip 1em 

We set $\Dy :=\{  A_n,B_n,C_n,D_n, E_6, E_7  \} $ and, for $X \in \Dy$, denote by $\Ix_{X}$ the set of all indices marked by $\circ$ in the above Dynkin diagram $\eqref{Eq: Dynkin}$ of type $X$. 

For a dominant minuscule weight $\La_i$, the set of vertices of the crystal $B(\La_i)$ can be identified with the Weyl group orbit $\weyl \cdot \La_i$.
Set $\cB(\La_i) := \weyl \cdot \La_i$ and define 
\begin{align*}
& \ep_i (\mu) := \max\{0, - \langle h_i, \mu \rangle \}, \qquad\qquad\quad  \ph_i (\mu) := \max\{0, \langle h_i, \mu \rangle \}, \\
& \tf_i(\mu) := 
\begin{cases}
\mu - \alpha_i  &\text{ if } \langle h_i, \mu \rangle=1,  \\
0 &\text{ otherwise, }
\end{cases}
\qquad 
\te_i(\mu) := 
\begin{cases}
\mu + \alpha_i  &\text{ if } \langle h_i, \mu \rangle=-1,  \\
0 &\text{ otherwise, }
\end{cases}
\end{align*}
for $\mu\in \cB(\La_i)$.
Then one can show that $\cB(\La_i)$ is a crystal and it is isomorphic to the crystal $B(\La_i)$.

\vskip 2em

\section{Cyclotomic quiver Hecke algebras and minuscule weights} \label{Sec: CQHA for ACDE}

\subsection{Quiver Hecke algebras} \label{Sec: QHA}

Let $\mathbf k$ be a field and $(\cmA,\wlP, \Pi,\wlP^{\vee},\Pi^{\vee})$ be a Cartan datum.
 Choose polynomials 
 $$
Q_{i,j}(u,v) =
\delta_{i,j}  \sum\limits_{ \substack{ (p,q)\in \Z^2_{\ge0} 
\\ (\alpha_i , \alpha_i)p+(\alpha_j , \alpha_j)q=-2(\alpha_i , \alpha_j)}}
t_{i,j;p,q} u^p v^q \in \bR [u,v]
$$
with $t_{i,j;p,q}\in{\mathbf k}$, $t_{i,j;p,q}=t_{j,i;q,p}$ and $t_{i,j:-a_{ij},0} \in {\mathbf k}^{\times}$.
Note that $Q_{i,j}(u,v)=Q_{j,i}(v,u)$ for $i,j\in I$. 
Let
${\mathfrak{S}}_{n} = \langle s_1, \ldots, s_{n-1} \rangle$ be the symmetric group
on $n$ letters with the action of  ${\mathfrak{S}}_n$ on $I^n$ by place permutation.
For  $\beta \in \rlQ_+$ with $\Ht(\beta) = n$, we set
$$I^{\beta} = \{\nu = (\nu_1, \ldots, \nu_n) \in I^{n}\mid \alpha_{\nu_1} + \cdots + \alpha_{\nu_n} = \beta\}.$$

\begin{defn}
For $\beta \in \rlQ_+$ with $\Ht(\beta)=n$, the \emph{quiver Hecke algebra}  
$R(\beta)$  associated
with $\{ Q_{i,j} \}_{i,j \in I}$ is the ${\mathbf k}$-algebra generated by $\{ e(\nu) \}_{\nu \in  I^{\beta}}$, $ \{x_k \}_{1 \le
k \le n}$, $\{ \tau_m \}_{1 \le m \le n-1}$ satisfying the following defining relations:
\begin{align*} 
& e(\nu) e(\nu') = \delta_{\nu, \nu'} e(\nu), \ \
\sum_{\nu \in  I^{\beta} } e(\nu) = 1, \allowdisplaybreaks\\
& x_{k} x_{m} = x_{m} x_{k}, \ \ x_{k} e(\nu) = e(\nu) x_{k}, \allowdisplaybreaks\\
& \tau_{m} e(\nu) = e(s_{m}(\nu)) \tau_{m}, \ \ \tau_{k} \tau_{m} =
\tau_{m} \tau_{k} \ \ \text{if} \ |k-m|>1, \allowdisplaybreaks\\
& \tau_{k}^2 e(\nu) = Q_{\nu_{k}, \nu_{k+1}} (x_{k}, x_{k+1})
e(\nu), \allowdisplaybreaks\\
& (\tau_{k} x_{m} - x_{s_k(m)} \tau_{k}) e(\nu) = \begin{cases}
-e(\nu) \ \ & \text{if} \ m=k, \nu_{k} = \nu_{k+1}, \\
e(\nu) \ \ & \text{if} \ m=k+1, \nu_{k}=\nu_{k+1}, \\
0 \ \ & \text{otherwise},
\end{cases} \allowdisplaybreaks\\
& (\tau_{k+1} \tau_{k} \tau_{k+1}-\tau_{k} \tau_{k+1} \tau_{k}) e(\nu)\\
& =\begin{cases} \dfrac{Q_{\nu_{k}, \nu_{k+1}}(x_{k},
x_{k+1}) - Q_{\nu_{k}, \nu_{k+1}}(x_{k+2}, x_{k+1})} {x_{k} -
x_{k+2}}e(\nu) \ \ & \text{if} \
\nu_{k} = \nu_{k+2}, \\
0 \ \ & \text{otherwise}.
\end{cases}
\end{align*}
\end{defn}

Note that $R(\beta)$ has the  $\Z$-grading defined by 
\begin{equation*} \label{eq:Z-grading}
\deg e(\nu) =0, \quad \deg\, x_{k} e(\nu) = (\alpha_{\nu_k}
, \alpha_{\nu_k}), \quad\deg\, \tau_{l} e(\nu) = -
(\alpha_{\nu_l} , \alpha_{\nu_{l+1}}).
\end{equation*}
For $n \in \Z_{\ge 0}$, we set $R(n):= \bigoplus_{\beta \in \rlQ_+, \ \Ht(\beta)=n} R(\beta)$.
For $w\in \sg_n$, we fix a preferred reduced expression $w = s_{i_1}\cdots s_{i_l} $ and define 
$ \tau_w := \tau_{s_{i_1}} \cdots \tau_{s_{i_l}}$. Note that $\tau_w$  depends on the choice of reduced expression $w$ unless $w$ is fully  commutative.

Throughout the paper, we assume that  
\begin{align} \label{Eq: Qij=1}
Q_{i,j}(u,v) = 1 \qquad \text{ for any $i,j\in I$ with $( \alpha_i, \alpha_j)=0$.}
\end{align}
If every element of the base field $\bR$ has a square root, then one can show that
any quiver Hecke algebra associated with $Q_{i,j}(u,v)$ is isomorphic to the quiver Hecke algebra associated with 
$$
Q_{i,j}'(u,v) = 
\begin{cases}
1  &\text{ if } (\alpha_i, \alpha_j)=0,  \\
Q_{i,j}(u,v) &\text{ otherwise, }
\end{cases}
$$
under the isomorphism given by $e(\nu) \mapsto e(\nu) $, $ x_k e(\nu) \mapsto  x_k e(\nu) $ and $\tau_{k} e(\nu) \mapsto c_{\nu_k, \nu_{k+1}} \tau_k e(\nu)$, where 
$$
c_{i,j} = 
\begin{cases}
( Q_{i,j}(u,v) )^{1/2}  &\text{ if } (\alpha_i, \alpha_j)=0,  \\
1 &\text{ otherwise. }
\end{cases}
$$
(See \cite[Lemma 3.2]{AIP15} for the isomorphism).

For a $\Z$-graded algebra $A$, we denote by $A \Mod$ the category of  graded left $A$-modules, 
 and write $A \proj$ (resp.\ $A \gmod$) for the full subcategory of $A \Mod$ consisting of
finitely generated projective (resp.\ finite-dimensional) graded $A$-modules. 
We set $R\proj := \bigoplus_{\beta \in \rlQ_+}R(\beta) \proj$ and 
$R\gmod := \bigoplus_{\beta \in \rlQ_+}R(\beta) \gmod$.

 For $M\in R(\beta)\Mod$ and $N\in R(\gamma)\Mod$, we define
$$
M\conv N =R(\beta+\gamma)e(\beta,\gamma) \otimes_{R(\beta) \otimes R(\gamma)} (M\boxtimes N),
$$
where $e(\beta,\gamma) =\displaystyle\sum_{ \nu_1 \in I^{\beta},  \nu_2 \in I^{\gamma} } e(\nu_1*\nu_2)$. Here $\nu_1*\nu_2$ is 
the concatenation of $\nu_1$ and $\nu_2$. 
Then, the Grothendieck groups $[R\proj]$ and $[R\gmod]$ admit $\A$-algebra structures with the multiplication given by the functor $\conv$.

\begin{thm} {\rm(\cite{KL09, KL11, R08})}\label{Thm: categorification}
There exist  $\A$-algebra isomorphisms 
\begin{eqnarray*}
 U_\A^-(\g) &\buildrel \sim \over \longrightarrow& [R \proj] , \qquad \Aqq \buildrel \sim \over \longrightarrow [R \gmod].
\end{eqnarray*}
\end{thm}

For $\Lambda \in \wlP_+$ and  $\beta \in \rlQ_+$, let $I^{\Lambda}_{\beta}$ be the two-sided ideal of $R(\beta)$  generated by the elements
$\{ x^{\langle h_{\nu_{\Ht(\beta)}}, \Lambda \rangle} e(\nu)  \mid \nu \in I^\beta \}$.
The \emph{cyclotomic quiver Hecke algebra} $R^\Lambda(\Lambda - \beta)$ is defined to be the quotient algebra 
$$
R^\Lambda(\Lambda - \beta):=R(\beta)/I^\Lambda_\beta.
$$
For $m\in \Z_{\ge 0}$, we set $R^\La(m) := \bigoplus_{\beta \in \rlQ_+,\ \Ht(\beta)=m} R^\La(\La-\beta)$.
When $\cmA$ is of type $X$, we sometimes write $R_X(m)$ and $R_X^\La(m)$ to emphasize the type $X$. 

For $M\in R^{\Lambda}(\mu)\Mod$, we set $\wt(M) = \mu$ and define functors $F_i^{\Lambda}$ and $ E_i^{\Lambda}$ 
by
\begin{align*}
F_i^{\Lambda}M := R^{\Lambda}(\mu-\alpha_i)e(\alpha_i,\beta)\otimes_{R^{\Lambda}(\mu)}M \quad \text{and} \quad
E_i^{\Lambda}M := e(\alpha_i,\beta-\alpha_i)M,
\end{align*}
where $\beta = \La-\mu$.
We denote by $\Ind_{m}^{m+1}$ and $\Res_m^{m+1}$ the induction and restriction functors via the embedding $R^\La(m) \hookrightarrow R^\La(m+1)$.
Note that $\Ind_{m}^{m+1} = \sum_{i\in I} F_i^{\Lambda}$ and $ \Res_m^{m+1} =  \sum_{i\in I} E_i^{\Lambda}$.
The functors $F_i^{\La}$ and $E_i^{\La}$ give a $U_\A(\g)$-module structure on $[R^{\Lambda} \proj]$ and $[R^{\Lambda} \gmod]$.

\begin{thm} {\rm (\cite{KK11})}
There exist $U_\A(\g)$-module isomorphisms 
\begin{eqnarray*}
[R^{\Lambda} \proj]  \buildrel \sim \over \longrightarrow V_\A(\Lambda),\qquad [R^{\Lambda} \gmod] \buildrel \sim \over \longrightarrow V_\A(\Lambda)^\vee.
\end{eqnarray*} 
\end{thm}

If $\xi$ is an extremal weight of $V(\La) $, then we know the representation type of $R^\La(\xi)$.
\begin{prop} {\rm(\cite[Corollary 4.7]{AP14}, \cite[Lemma 1.11]{KKOP18}) } \label{Prop: simple}
Let $\La \in \wlP_+$ and $\xi = w\La $ for $w\in \weyl$.
Then the algebra $R^{\La}(\xi)$ is simple.
\end{prop}

\begin{prop} {\rm(\cite[Corollary 3.19]{KL09}) } \label{Prop: abs irr}
Every irreducible $R$-module is absolutely irreducible.
\end{prop}

We now assume that $\cmA$ is of simply laced type.
Then the construction for homogeneous representations was given in \cite{KR10}.
We denote by $G_\beta$ the graph with the set of vertices $I^\beta$, and with $\nu,\mu \in I^\beta$ connected by an edge if and only if 
there exists $k $ such that $\mu = s_k\nu$ and $ a_{\nu_{k}, \nu_{k+1}}=0$. Let $C$ be a connected component of $G_\beta$.
We say that $C$ is \emph{homogeneous} if for each $\nu \in C$ the following condition holds:
\begin{equation} \label{Eq: homo}
\begin{aligned}
&\text{ if $\nu_r = \nu_s$ for some $r<s$ then there exist $t,u$ with $ r < t < u < s $}\\
&\text{ such that $a_{\nu_r, \nu_t}= -1$ and $a_{\nu_u, \nu_s}= -1 $.}
\end{aligned}
\end{equation}
It was proved in \cite[Theorem 3.4]{KR10} that $C$ has a simple $R(\beta)$-module structure.

\subsection{Minuscule representations of symmetric type $ADE$} \label{Sec: minuscule ADE}

Let $V(\La_i)$ be a minuscule representation and $B(\La_i)$ its crystal. 
Let  $b\in B(\La_i)$, $\xi = \wt(b)$ and $\ell = \Ht(\La_i - \xi)$. 
We set 
\begin{align} \label{Eq: paths}
\Pa(b) := \{ (i_1, \ldots, i_\ell) \in I^\ell \mid \tf_{i_1}\cdots \tf_{i_\ell} b_{\La_i} = b    \},
\end{align}
where $b_{\La_i}$ is the highest weight vector of $B(\La_i)$. From a viewpoint of the crystal graph $B(\La_i)$,
an element of $\Pa(b)$ can be thought as a sequence of arrows from $b_{\La_i}$ to $b$ on the graph $B(\La_i)$.
In this sense, an element of $\Pa(b)$ is called a \emph{path} from $b_{\La_i}$ to $b$ on the crystal $B(\La_i)$.

From now on, we suppose that  $X \in \{ A_n, D_n, E_6, E_7   \}$ and $i\in \Ix_X$.

\begin{lem} \label{Lem: reduced exp}
Let $ (i_1, \ldots, i_\ell) \in  \Pa(b)$ and set $w = s_{i_1} s_{i_2} \cdots s_{i_\ell} \in \weyl$.
Then $s_{i_1} s_{i_2} \cdots s_{i_\ell}$ is a reduced expression of $w$.
\end{lem}
\begin{proof}
 As $(i_1, \ldots, i_\ell) \in  \Pa(b)$, we have $ \wt(b) = \La_i - \sum_{k=1}^\ell \alpha_{i_k}  $ by the definition of $\Pa(b)$.
On the other hand, since $B(\La_i)$ is the crystal of the minuscule representation $V(\La_i)$, we have $ \wt(b) = w \La_i $,
which says that 
$$
\La_i  - w\La_i  = \sum_{k=1}^\ell \alpha_{i_k}.
$$
Thus we have 
\begin{align} \label{Eq: Ht ell}
\Ht(\La_i - w\La_i) = \ell. 
\end{align}

We now set $m := \ell(w)$ and take a reduced expression $ s_{j_1} \ldots s_{j_m} $ of $w$. For $k=1, \ldots, m$, define 
$ w_k := s_{j_k} s_{j_{k+1}} \cdots s_{j_m} $ and  $w_{m+1} := \id$. 
Then we have 
\begin{align*}
\La_i - w \La_i  &= \sum_{k=1}^m ( w_{k+1}\La_i - w_k \La_i ) = \sum_{k=1}^m ( w_{k+1}\La_i - s_{j_k} (w_{k+1} \La_i) ) \\
& =  \sum_{k=1}^m  \langle  h_{j_k},  w_{k+1} \La_i \rangle \alpha_{j_k}.
\end{align*}
Since $\La_i$ is minuscule, we have 
$  \langle  h_{j_k},  w_{k+1} \La_i \rangle =  \langle w_{k+1}^{-1} h_{j_k},   \La_i \rangle \in \{ 0, \pm1 \}$ for any $k$,
which implies that 
\begin{align} \label{Eq: max m}
\Ht(\La_i - w \La_i) \le \sum_{k=1}^m |  \langle  h_{j_k},  w_{k+1} \La_i \rangle | \le m = \ell(w).
\end{align}
Combining $\eqref{Eq: max m}$ with $ \eqref{Eq: Ht ell}$, we obtain 
$$
\ell \le \ell(w),
$$
which tells us that $s_{i_1} s_{i_2} \cdots s_{i_\ell}$ is a reduced expression of $w$.
\end{proof}

\begin{lem} \label{Lem: w(b)}
For $ (i_1, \ldots, i_\ell),\ (j_1, \ldots, j_\ell) \in  \Pa(b)$, we have 
$$
s_{i_1} \cdots s_{i_\ell} = s_{j_1} \cdots s_{j_\ell}.
$$
\end{lem}
\begin{proof}
Let $ (i_1, \ldots, i_\ell) \in  \Pa(b)$ and set $w := s_{i_1} \cdots s_{i_\ell}$.
Thanks to Lemma \ref{Lem: reduced exp}, $s_{i_1} \cdots s_{i_\ell}$ is a reduced expression of $w$.
Since $B(\La_i)$ is minuscule, $s_i b' = \tf_i b'$ for any $b'\in B(\La_i)$ with $\varphi_i(b')>0$. Thus, 
the reduced expression $w = s_{i_1} \cdots s_{i_\ell}$ satisfies $\eqref{Eq: def of minuscule}$, i.e. $w$ is $\La_i$-minuscule.
 It was shown in \cite[Proposition 2.1]{St01}  that every reduced expression of $w$
satisfies $\eqref{Eq: def of minuscule}$. We take $j\in I$ with $j \ne i$. 

We suppose that $ \ell(w s_j) < \ell(w) $.
Then one can write a reduced expression of $w = s_{t_1} \cdots s_{t_{\ell-1}} s_j$ for some $t_1, \ldots, t_{\ell} \in I $. 
Since $\langle h_j, \La_i \rangle=0$, $w = s_{t_1} \cdots s_{t_{\ell-1}} s_j$ is not $\La_i$-minuscule which is a contradiction. Thus, we have
\begin{align} \label{Eq: mcr w}
\text{$\ell(w s_j) > \ell(w) $ for all $j\in I$ with $j \ne i$}.
\end{align}

We now consider an element $(j_1, \ldots, j_\ell) \in  \Pa(b)$. Set $u := s_{j_1} \cdots s_{j_\ell}$.
By the same argument as above, 
\begin{align} \label{Eq: mcr u}
\text{$\ell(u s_j) > \ell(u) $ for all $j\in I$ with $j \ne i$}
\end{align}
Considering the realization $\cB(\La_i)$ given in Section \ref{Sec: minuscule rep}, we have $ w \cdot \La_i = u \cdot \La_i = \wt(b) $, which tells us that 
$$
u^{-1} w\cdot \La_i = \La_i.
$$
Hence $ u^{-1} w$ is contained in the parabolic subgroup $\weyl_J$ of $\weyl$ generated by $s_j$ for $j\in J := I\setminus \{ i \}$. 
We set $X_J := \{ w\in \weyl \mid \ell(w s_j) > \ell(w) \text{ for all } j\in J \} $. Note that $u,w \in X_J$ by $\eqref{Eq: mcr w}$ and $\eqref{Eq: mcr u}$.
It is well-known that $ \ell(xw) = \ell(x) + \ell(w) $ for $x\in X_J$ and $w\in W_J$ (see \cite[Proposition 2.1.1]{GP00} for example).

Suppose that $u^{-1} w \ne \mathrm{id}$. Then there exists $j \in I \setminus \{ i \}$ such that $ \ell(u^{-1} w s_j) = \ell(u^{-1} w)-1$.
It follows from $\eqref{Eq: mcr u}$ that 
\begin{align*}
\ell(w s_j) &= \ell(u (  u^{-1} w  s_j )) = \ell(u) + \ell( (u^{-1} w) s_j ) \\
&= 
\ell(u) + \ell( u^{-1} w ) - 1 = \ell(u (u^{-1} w)) -1 = \ell( w) -1,
\end{align*}
which is a contradiction to $\eqref{Eq: mcr w}$. 
The second and forth identities follow from the fact that $u \in X_J$ and $ u^{-1} w, (u^{-1} w) s_j \in \weyl_J  $.
Therefore, we conclude that $u^{-1} w = \mathrm{id}$.
\end{proof}

We define 
$$
w(b) := s_{i_1} \cdots s_{i_\ell} \in \weyl,
$$
where $(i_1, \ldots, i_{\ell})$ is an element of $\Pa(b)$.
Note that it is well-defined by Lemma \ref{Lem: w(b)}.

\begin{lem} \label{Lem: 1-1 cor}
For $b\in B(\La_i)$, the set $\Pa(b)$ is in 1-1 correspondence to the set of all reduced expressions of $w(b)$ via the map $(i_1, \ldots, i_\ell) \mapsto s_{i_1} \cdots s_{i_\ell}$.
\end{lem}
\begin{proof}
Let $E(b)$ be the set of all reduced expressions of $w(b)$.  
We consider the map
$$
F : \Pa(b) \longrightarrow E(b), \qquad (i_1, \ldots, i_\ell) \mapsto s_{i_1} \cdots s_{i_\ell}.
$$
Thanks to Lemma \ref{Lem: reduced exp} and Lemma \ref{Lem: w(b)}, the map $F$ is well-defined. 

Recall the crystal realization $\cB(\La_i)$ given in Section \ref{Sec: minuscule rep}. Note that, for $\mu\in \cB(\La_i)$, $  \tf_i(\mu) = s_i \mu$ if $\langle h_i, \mu \rangle=1$.
Since every reduced expression of $w(b)$ satisfies $\eqref{Eq: def of minuscule}$ by \cite[Proposition 2.1]{St01}, the map 
$$
G : E(b) \longrightarrow \Pa(b), \qquad s_{i_1} \cdots s_{i_\ell} \mapsto (i_1, \ldots, i_\ell), 
$$
is well-defined.
Then, it is clear that $F\circ G = \mathrm{id}$ and $G\circ F = \mathrm{id}$.
\end{proof}

Since $w(b)$ is fully commutative \cite[Proposition 2.1]{St01}, Lemma \ref{Lem: 1-1 cor} says that for $\nu, \mu \in \Pa(b)$, there exists $w_{\nu, \mu}\in \sg_\ell$ such that $ \nu = w_{\nu, \mu} \mu$.
We set 
\begin{align} \label{Eq: c for ADE}
c_{\nu, \mu} := e(\nu) \tau_{w_{\nu, \mu}}e(\mu).
\end{align}
Note that $c_{\nu, \mu}$ does not depend on the choice of the reduced expressions of $w_{\nu, \mu}$.

We now set 
$$ 
\Sp(b) := \bigoplus_{\nu \in \Pa(b)} \bR \bp(\nu) 
$$
 and define an $R(\beta)$-module structure on $\Sp(b)$ by  
$$
e(\mu) \bp(\nu) = \delta_{\nu, \mu} \bp(\nu) , \quad x_i e(\mu) \bp(\nu) = 0, \quad  
\tau_j e(\mu) \bp(\nu) =
\begin{cases}
\bp(s_j \nu) ,  &\text{ if } \nu = \mu \text{ and } s_j \nu \in \Pa(b),  \\
0 &\text{ otherwise,}
\end{cases}  
$$
for $\nu \in \Pa(b)$ and admissible $i,j, \mu $.

\begin{thm} \label{Thm: main ADE}
Let $X \in \{ A_n, D_n, E_6, E_7 \}$ and $i\in \Ix_X$.
Let $b\in B(\La_i)$ and write $ \xi = \wt(b)$.
\begin{enumerate}
\item $\Sp(b)$ is a homogeneous simple $R^{\La_i}(\xi)$-module.
\item 
$R^{\La_i}(\xi)$ has a $\bR$-basis 
$$
\Cb(b):= \{ c_{\nu, \mu}  \mid \nu, \mu \in \Pa(b) \},
$$
where $c_{\nu, \mu}$ is given in $\eqref{Eq: c for ADE}$.
\end{enumerate}
\end{thm}
\begin{proof}
(1) By Lemma \ref{Lem: 1-1 cor}, the set $\Pa(b)$ can be identify with the set of all reduced expression of $w(b)$.
Then it follows from \cite[Proposition 2.5]{St01} (c.f. \cite[Proposition 3.8]{KR10}) that 
$\Pa(b)$ satisfies the condition $\eqref{Eq: homo}$. Thus, 
$\Pa(b)$ has a homogeneous simple $R$-module structure which is same as the above actions on $\Sp(b)$. 
Moreover, since $ \nu_\ell = i$ for any $ ( \nu_1, \ldots, \nu_\ell) \in \Pa(b)$ and $x_\ell$ acts as zero, 
$\Sp(b)$ is an $R^{\La_i}(\xi)$-module.

(2) Since every element of $B(\La_i)$ is extremal, $R^{\La_i}(\xi)$ is simple by Proposition \ref{Prop: simple}. 
Thus, by Proposition \ref{Prop: abs irr}, $R^{\La_i}(\xi)$ is isomorphic to the endomorphism ring of $\Sp(b)$, which implies that $\Cb(b)$ is a basis of $R^{\La_i}(\xi)$.
\end{proof}

\begin{Rmk}
For type $A_n$, the simple module $\Sp(b)$ can be realized as the set of standard tableaux of a usual Young diagram (see \cite[Section 5]{BK09}). 
Thus, the algebra $R^{\La_i}(m)$ has a $\bR$-basis indexed by all pairs of standard tableaux.
In Section \ref{Sec: Irreducible repn}, we shall give a $B_n$-type analogue using shifted Young tableaux.
\end{Rmk}

\subsection{Minuscule representations of type $C_n$}
In this subsection, we assume that $X =C_n$ and consider the dominant minuscule weight $\La_n$.
Then the crystal $B(\La_n)$ is given as follows:
$$
\CrystalC
$$

For $b\in B(\La_n)$, if we write $b = \tf_{i_1} \cdots \tf_{i_{\ell}} b_{\La_n}$, then we set 
$$
\nu_b := ( i_1, \ldots, i_\ell).
$$
We set $\Sp(b) := \bR \bp(\nu_b)$ and define an $R$-module structure by 
$$
e(\mu) \bp(\nu_b) = \delta_{\nu_b, \mu} \bp(\nu_b) , \quad x_i e(\mu) \bp(\nu_b) = \tau_j e(\mu) \bp(\nu_b) = 0,
$$
for admissible $i,j,\mu$. Then one can prove that $\Sp(b)$ is a 1-dimensional $R^{\La_n}$-module.
Note that this construction of $\Sp(b)$ is the same as the one for type $ADE$.
Proposition \ref{Prop: simple} and Proposition \ref{Prop: abs irr} imply the following immediately.
\begin{prop} \label{Prop: main C} 
Let $X =C_n$, $b\in B(\La_n)$, and write $ \xi = \wt(b)$. Then
$R^{\La_n}(\xi)$ is isomorphic to $\bR$.
\end{prop}

\vskip 2em

\section{Simple $R^{\Lambda_1}$-modules of type $B_n$}  \label{Sec: Irreducible repn}

\subsection{Shifted Young tableaux}  \label{Sec: Shifted Young tableaux}

For  a \emph{strict partition} $ \la = (\lambda_1 > \lambda_2 >  \cdots > \la_\ell > 0 )$ with $N = \sum_{ i=1}^\ell \la_i$, we write $\la \vdash N$ and set $|\la| := N $.
We identify strict partitions with shifted Young diagrams and depict shifted Young diagrams using English convention.
We write $b=(i,j) \in \lambda$ if there exists a box in the $i$th row and the $j$th column.
By an \emph{addable} (resp.\ a \emph{removable}) box $b\in \la$, we mean a box which can be added to (resp.\ removed from) $\la$ to obtain a valid shifted Young diagram $\la \swarrow b$ (resp.\ $\la\nearrow b$).
For $n\in \Z_{>0}$, let $\Par_n$ be the set of all strict partitions $(\la_1, \la_2, \ldots )$ such that $\la_1 \le n$. It is easy to check that $| \Par_0 |=1  $ and 
$ |\Par_n| = 2| \Par_{n-1}| $ for $n \ge  1$, which say that 
\begin{align} \label{Eq: size of Par}
| \Par_n |  = 2^{n} \qquad \text{ for } n \ge 1.
\end{align}

A {\it standard tableau} $T$ of shape $\lambda\vdash n$ is
a filling of boxes of $\lambda$ with $\{ 1,2, \ldots, n\} $ such that (i) each entry is used exactly once, (ii) the entries in rows
and columns increase from left to right and top to bottom, respectively. Let $\ST(\lambda)$ be the set of all standard tableaux of shape $\lambda$. 
For example, the following is a standard tableaux of shape $\la = (7,4)$.
\vskip 0.3em
$$
\xy
(0,12)*{};(42,12)*{} **\dir{-};  
(0,6)*{};(42,6)*{} **\dir{-};
(6,0)*{};(30,0)*{} **\dir{-};
(0,6)*{};(0,12)*{} **\dir{-};
(6,0)*{};(6,12)*{} **\dir{-};
(12,0)*{};(12,12)*{} **\dir{-};
(18,0)*{};(18,12)*{} **\dir{-};
(24,0)*{};(24,12)*{} **\dir{-};
(30,0)*{};(30,12)*{} **\dir{-};
(36,6)*{};(36,12)*{} **\dir{-};
(42,6)*{};(42,12)*{} **\dir{-};
(3,9)*{1}; (9,9)*{2}; (15,9)*{4}; (21,9)*{5}; (27,9)*{7}; (33,9)*{8}; (39,9)*{11};
(9,3)*{3}; (15,3)*{6}; (21,3)*{9}; (27,3)*{10};
\endxy
$$

Let $\lambda \vdash n$ be a shifted Young diagram. We define the residue of $(i,j) \in \la$ as follows:
\begin{align} \label{Eq: res}
\res(i,j) := j-i+1.
\end{align}
For example, the residues for $\lambda = (9,6,3,1)$ are as follows. In this example,  $\res(2,6) = 5 $.   
$$
\xy
(0,12)*{};(54,12)*{} **\dir{-};
(0,6)*{};(54,6)*{} **\dir{-};
(6,0)*{};(42,0)*{} **\dir{-};
(12,-6)*{};(30,-6)*{} **\dir{-};
(18,-12)*{};(24,-12)*{} **\dir{-};
(0,12)*{};(0,6)*{} **\dir{-};
(6,12)*{};(6,0)*{} **\dir{-};
(12,12)*{};(12,-6)*{} **\dir{-};
(18,12)*{};(18,-12)*{} **\dir{-};
(24,12)*{};(24,-12)*{} **\dir{-};
(30,12)*{};(30,-6)*{} **\dir{-};
(36,12)*{};(36,0)*{} **\dir{-};
(42,12)*{};(42,0)*{} **\dir{-};
(48,12)*{};(48,6)*{} **\dir{-};
(54,12)*{};(54,6)*{} **\dir{-};
(3,9)*{1}; (9,9)*{2}; (15,9)*{3}; (21,9)*{4}; (27,9)*{5}; (33,9)*{6}; (39,9)*{7};(45,9)*{8};(51,9)*{9}; 
           (9,3)*{1}; (15,3)*{2}; (21,3)*{3}; (27,3)*{4}; (33,3)*{5}; (39,3)*{6};
                      (15,-3)*{1}; (21,-3)*{2}; (27,-3)*{3};
                                   (21,-9)*{1};
\endxy
$$
\begin{Rmk}
The above residues can be viewed as the limit of the residues given in \cite{AP14} as $\ell$ goes to infinity. 
\end{Rmk}

For  $T \in \ST(\lambda)$, we define the \emph{residue sequence} of $T$ by
$$ \res(T) = (  \res(i_n,j_n), \ldots, \res(i_2,j_2) , \res(i_1,j_1) )\in I^n,$$
where $(i_k,j_k)\in\lambda$ is the box filled with $k$ in $T$.

\subsection{Irreducible representations}

For $\la \in \Par_n$, we set 
$$
\wt(\la) := \La_1 - \sum_{b\in \la} \alpha_{\res(b)} \in \wlP.
$$
Recall the crystal realization $\cB(\La_i)$ given in Section \ref{Sec: minuscule rep}.
\begin{lem} \label{Lem: bijection}
The map $ \la \mapsto  \wt(\la)$ gives 1-1 correspondence between $\Par_n$ and $\cB(\La_1)$. 
\end{lem}
\begin{proof}
Let $f:\Par_n \longrightarrow \wlP $ be the map defined by 
$$
f(\la) = \wt(\la) \quad \text{ for } \la \in \Par_n.
$$
Let $\la \in \Par_n$.
We first show that $f(\la) \in \weyl\cdot \La_1$ by induction argument on $|\la|$.  
Since $\wt(\emptyset) = \La_1 \in \weyl \cdot \La_1$, we assume that $|\la| > 0$.
Thus we can write $\la = (\la_1 > \la_2> \cdots > \la_\ell > 0) \in \Par_n$, and set
$$
\beta_k := \sum_{t=1}^{k} \alpha_t \quad \text{ for } k=1, \ldots, n.
$$
Then we have $\wt(\la) = \La_1 - \sum_{k=1}^\ell \beta_{\la_k}$ and 
$$
\langle  h_i, \beta_{\la_j} \rangle = 
\begin{cases}
2  &\text{ if } i= \la_j =1,  \\
1  &\text{ if } i=\la_j >1,  \\
-1  &\text{ if } i=\la_j +1,  \\
0 &\text{ otherwise.}
\end{cases}  
$$
Let $b := ( \ell, \la_\ell + \ell - 1) \in \la$ and $\mu := \la \nearrow b$. By the induction hypothesis, we have $ \wt(\mu) \in \weyl\cdot \La_1$.
Setting $i:= \res(b)$, we obtain 
\begin{align*}
\langle h_i, \wt(\mu) \rangle = \langle h_i, \La_1 \rangle - \sum_{k=1}^\ell \langle h_i, \beta_{\la_k} \rangle + \langle h_i, \alpha_i \rangle = 1,
\end{align*}
which tells us that 
$$
\wt(\la) = \wt(\mu) - \alpha_i = s_i( \wt(\mu) ) \in \weyl\cdot \La_1.
$$
Thus, we have that $ \im f \subset \weyl\cdot \La_1$.

We now show that $f$ is injective.  Let $ \la =(\la_1, \la_2 , \ldots), \mu = ( \mu_1, \mu_2, \ldots ) \in \Par_n $, and assume that 
$\la \ne \mu$, i.e., there exists $k$ such that $ \la_i = \mu_i $ for $i < k$ and $\la_k \ne \mu_k$. 
Then by the definition, we have 
$$
\max\{\la_k, \mu_k\} \in  \mathrm{supp} ( \beta_{\la_k} - \beta_{\mu_k} )\quad \text{ and } \quad  \max\{\la_k, \mu_k\} \notin  \mathrm{supp} ( \beta_{\la_t} - \beta_{\mu_t} ) \text{ for } t > k,
$$ 
where $\mathrm{supp}(\beta) := \{ i \in I \mid b_i \ne 0  \} $ for $\beta = \sum_{i\in I} b_i \alpha_i \in \rlQ$.
This tells us that 
$$ 
\wt(\la) - \wt(\mu)  = \beta_{\la_k} - \beta_{\mu_k} + \sum_{t>k} (\beta_{\la_t} - \beta_{\mu_t}) \ne 0.
$$
Hence $f$ is injective. 

Moreover, if we denote by $P$ the subgroup of $\weyl$ generated by $s_2, \ldots, s_n $, then 
by $\eqref{Eq: size of Par}$, we have 
$$
|\Par_n| = 2^n = \frac{ 2^n n! }{ n! } = \left | \frac{\weyl} { P}  \right | = |\weyl \cdot \La_1 |. 
$$
Therefore, the map $f$ gives a bijection between $ \Par_n$ and $\weyl\cdot \La_1$.
\end{proof}

For $b \in B(\La_1)$, we denote by $\la_b$ the strict partition corresponding to $b$ under the bijection given in Lemma \ref{Lem: bijection}.
We set 
\begin{align*} 
\Sp(b) := \bigoplus_{T \in \ST(\la_b)} \bR T,
\end{align*}
and define an $R(\wt(b))$-module structure on $\Sp(b)$ by  
\begin{equation} \label{Eq: Sp B}
\begin{aligned}
e(\nu) T &= \delta_{ \nu, \res(T)} T,\\
 x_i e(\nu) T &= 0, \\
\tau_j e(\nu) T &=
\begin{cases}
s_j T  &\text{ if } \nu = \res(T) \text{ and } s_j T \in \ST(\la_b),  \\
0 &\text{ otherwise,}
\end{cases}  
\end{aligned}
\end{equation}
for $T \in \ST(\la_b)$ and admissible $i,j, \nu $. 
Here, $s_j T$ is the tableau obtained from $T$ by exchanging the entries $j$ and $j+1$.
For $k \in \Z_{\ge0}$, we simply write $\Sp(k)$ for the simple module $\Sp(b)$ where 
$b$ is the element of $B(\La_1)$ such that $\la_b = (k)$.

\begin{Rmk} \label{Rmk: ST}
One can describe the $U_q(B_n)$-crystal structure on $\Par_n$ directly  
via the 1-1 correspondence given in Lemma \ref{Lem: bijection}. 
In this description, the crystal operators $\tf_i$ and $\te_i$ are defined as follows: for $i\in I$ and $\la \in \Par_n$, 
\begin{align*}
\tf_i \la &= 
\begin{cases}
\la \swarrow b  &\text{ if $b$ is an addable box of $\la$ with residue $i$},  \\
0 &\text{ otherwise,}
\end{cases}  \\
\te_i \la &= 
\begin{cases}
\la \nearrow b  &\text{ if $b$ is a removable box of $\la$ with residue $i$},  \\
0 &\text{ otherwise.}
\end{cases}  
\end{align*}
Thus, for $b\in B(\La_1)$, we can identify $\ST(\la_b)$ with the set $\Pa(b)$ of all paths from $b_{\La_1}$ to $b$ on the crystal $B(\La_1)$.
This crystal realization also can be obtained from the $B_\infty$-Fock space given in \cite{JM83} by restricting to type $B_n$.
From the viewpoint of crystal, if the formulas $\eqref{Eq: Sp B}$ were written in terms of $\Pa(b)$, then it is same as the ones for type $ADE$ and $C$.
\end{Rmk}

For $S,T \in \ST(\la_b)$, let $w_{T, S}$ be an element of $\sg_{|\la_b|}$ such that $T = w_{T,S}S$, and set 
\begin{align} \label{Eq: c for B}
c_{T,S} := e(\res(T)) \tau_{w_{T,S}} e(\res(S)).
\end{align}

\begin{thm} \label{Thm: main B}
Let $X = B_n$, $b\in B(\La_1)$, and write $ \xi = \wt(b)$.
\begin{enumerate}
\item $\Sp(b)$ is a homogeneous simple $R^{\La_1}(\xi)$-module.
\item If we write $\la_b = (\la_1, \ldots, \la_\ell)$, then 
$$
\Sp(b) \simeq \hd( \Sp(\la_\ell) \conv \cdots \Sp(\la_2) \conv \Sp(\la_1)) ,
$$
where $hd (M)$ denotes by the head of a module $M$.
\item 
Let $m = \Ht( \La_1 -  \xi)$. Then 
$$
\Ind_{m}^{m+1} \Sp(b) \simeq \bigoplus_{i\in I, \tf_i b \ne 0} \Sp(\tf_i b),\qquad 
\Res^{m}_{m-1} \Sp(b) \simeq \bigoplus_{i\in I, \te_i b \ne 0} \Sp(\te_i b),
$$
where $\Ind_{m}^{m+1}$ and $\Res^{m}_{m-1}$ are the induction and restriction functors given in Section  \ref{Sec: QHA}. 
\item 
For $T,S \in \ST(\la_b)$, $c_{T, S}$ does not depend on the choice of the reduced expressions of $w_{T, S}$.
Moreover, the algebra
$R^{\La_1}(\xi)$ has a $\bR$-basis 
$$
\Cb(b):= \{ c_{T, S}  \mid T, S \in \ST(\la_b) \}.
$$
\end{enumerate}
\end{thm}
\begin{proof}

(1)
For $i<j$,
by the assumption $\eqref{Eq: Qij=1}$,  we can write 
$$
Q_{i,j}(u,v) =
\begin{cases}
\mathsf{a}_{i,j}u^2+ \mathsf{b}_{i,j} v  &\text{ if } i=1, j=2,  \\
\mathsf{a}_{i,j}u + \mathsf{b}_{i,j}v  &\text{ if }  j=i+1 > 2,  \\
1 &\text{ otherwise,}
\end{cases}  
$$ 
for some nonzero $\mathsf{a}_{i,j}, \mathsf{b}_{i,j} \in \bR^\times$.
Note that, if every element of the base field $\bR$ has a square root, then
the isomorphism given in \cite[Lemma 3.2]{AIP15} allows us to assume that $\mathsf{a}_{i,j}=1$ and $\mathsf{b}_{i,j}=-1$.

Let $T \in \ST(\la_b)$, $N := |\la_b|$ and write $\res(T) = ( i_1, \ldots,i_N )$.
Considering the residue pattern $\eqref{Eq: res}$ and combinatorics of standard tableaux,  we have
\begin{enumerate}
\item[(i)] $ i_k \ne i_{k+1} $  for any $k=1, \ldots, N-1$,
\item[(ii)] $s_k T$ is standard if and only if $a_{i_{k}, i_{k+1}}=0$, 
\item [(iii)] if $i_k = i_{k+2}$, then $( i_k, i_{k+1}, i_{k+2} ) = (1,2,1)$.
\end{enumerate}
We shall only check that the defining relations for  $( \tau_k x_m - x_{s_k(m)}\tau_k)$, $ \tau_k^2 $ and $ \tau_{k+1}\tau_{k}\tau_{k+1}  - \tau_{k}\tau_{k+1}\tau_{k}$
since it is straightforward to check other relations.

By (i), it suffices to show that  
$
( \tau_k x_m - x_{s_k(m)}\tau_k)e(\nu) = 0,
$
which always holds by the definition of $\Sp(b)$.

If $ a_{i_{k}, i_{k+1}}\ne0$, then  $\deg Q_{i_k, i_{k+1}}(u,v) > 0$ and  $s_k T \notin \ST(\la_b)$ by (ii). Thus we have 
$$
\tau_k^2 e(\nu) T = 0 = Q_{i_k, i_{k+1}}( x_k, x_{k+1}) e(\nu) T.
$$
If $ a_{i_{k}, i_{k+1}}=0$, then $Q_{i_k, i_{k+1}}(u,v)=1$ and 
$$
\tau_k^2 e(\nu) T =  \delta_{\nu, \res(T)} T = Q_{i_k, i_{k+1}}( x_k, x_{k+1})  e(\nu)T
$$
by (ii). Thus, the defining relation for $ \tau_k^2 e(\nu) $ holds. 

We now consider the last case. If  $i_k \ne i_{k+2}$, then it is easy to check that $ \tau_{k+1} \tau_k \tau_{k+1} e(\nu) T =  \tau_{k} \tau_{k+1} \tau_{k} e(\nu) T $.
Suppose that $i_k = i_{k+2}$. Then (iii) says that $( i_k, i_{k+1}, i_{k+2} ) = (1,2,1)$. Thus, by the definition of $\Sp(b)$, 
we have 
$$
( \tau_{k+1}\tau_{k}\tau_{k+1}  - \tau_{k}\tau_{k+1}\tau_{k}) e(\nu) T = 0.
$$
On the other hand, since $( i_k, i_{k+1}, i_{k+2} ) = (1,2,1)$,  we obtain
$$
\dfrac{Q_{1, 2}(x_{k},
x_{k+1}) - Q_{1, 2}(x_{k+2}, x_{k+1})} {x_{k} -
x_{k+2}}
= \mathsf{a}_{1,2}( x_k + x_{k+2} ),
$$
which implies that 
$$
\dfrac{Q_{1, 2}(x_{k},
x_{k+1}) - Q_{1, 2}(x_{k+2}, x_{k+1})} {x_{k} -
x_{k+2}}e(\nu) T = 0.
$$
Therefore, the defining relations for  $ \tau_{k+1}\tau_{k}\tau_{k+1}  - \tau_{k}\tau_{k+1}\tau_{k}$ holds, which completes the proof.

(2) We set $e := e(\nu_{\la_\ell} * \ldots * \nu_{\la_2} * \nu_{\la_1})$ where $ \nu_k = (k, \ldots, 2,1)$ for $k\in \Z_{\ge 0}$. Considering the combinatorics of shifted Young tableaux, it is easy to see that 
$ e \Sp(b) = \bR T_0$, where $T_0$ is the \emph{initial tableau}, which is filled with $\{1, \ldots, N\}$ from left to right and top to bottom in order.
Thus, there exists a surjective homomorphism 
$$
\Sp(\la_\ell) \otimes \cdots \otimes \Sp(\la_2) \otimes \Sp(\la_1) \twoheadrightarrow e\Sp(b),
$$
which yields the surjective homomorphism
\begin{align} \label{Eq: sur}
\Sp(\la_\ell) \conv \cdots \conv \Sp(\la_2) \conv \Sp(\la_1) \twoheadrightarrow \Sp(b).
\end{align}

On the order hand, by the shuffle lemma \cite[Lamma 2.20]{KL09},  we have 
$$ \dim  e(  \Sp(\la_\ell) \conv \cdots \conv \Sp(\la_2) \conv \Sp(\la_1) ) =1.$$ 
For any quotient $Q$ of $\Sp(\la_\ell) \conv \cdots \conv \Sp(\la_2) \conv \Sp(\la_1)$, the natural surjection 
$$\Sp(\la_\ell) \conv \cdots \conv \Sp(\la_2) \conv \Sp(\la_1) \twoheadrightarrow Q$$
tells us that $ \dim (e Q)=1 $. Thus, we conclude that $\Sp(\la_\ell) \conv \cdots \conv \Sp(\la_2) \conv \Sp(\la_1) $
has a unique simple head. Hence, $\eqref{Eq: sur}$ implies the desired result.

(3) Note that every element of $B(\La_1)$ is extremal.  By Proposition \ref{Prop: simple}, $R^{\La_1}(\zeta)$ is simple for any weight $\zeta$ with $B(\La_1)_\zeta \ne 0$, which says that $R^{\La_1}(t)$ is semisimple for any $t > 0$. 
Since $\Sp( b)$ has a basis indexed by $\ST(\la_b)$, we have $E_i^{\La_1} \Sp(b) \simeq \Sp(\te_i b) $ by considering their characters. Here we set $\Sp(0):=0$.
Thus, we have $\Res^{m}_{m-1} \Sp(b) \simeq \bigoplus_{i\in I, \te_i b \ne 0} \Sp(\te_i b)$.
The isomorphism for $\Ind^{m+1}_m$ follows from the fact that $F_i^\La$ and $E_i^\La$ form an adjoint pair \cite[Theorem 3.5]{Kas12}.

(4) It follows from the proof of (1) that $\tau_{k+1} \tau_k \tau_{k+1} - \tau_{k} \tau_{k+1} \tau_k$ acts as 0 on $\Sp(b)$. 
Thus, for any $T \in \Sp(b)$, $\tau_w T$ does not depend on the choice of the reduced expressions of $w \in \sg_N$.
By Proposition \ref{Prop: simple} and Proposition \ref{Prop: abs irr},  $R^{\La_1}(\xi)$ is isomorphic to the endomorphism ring of $\Sp(b)$.
Therefore,  for $T,S \in \ST(\la_b)$, $c_{T, S}$ does not depend on the choice of the reduced expressions of $w_{T, S}$ and
$\Cb(b)$ is a basis of $R^{\La_1}(\xi)$.
\end{proof}

\begin{cor} \label{Cor: isom}
 For $n\in \Z_{\ge 2}$ and $m\in \Z_{\ge0}$, there exist algebra isomorphisms
$$
R^{\La_1}_{B_n}(m) \simeq R^{\La_1}_{D_{n+1}}(m) \simeq R^{\La_2}_{D_{n+1}}(m).
$$
In particular, $ \dim  R^{\La_1}_{B_n}(m)=  \dim  R^{\La_1}_{D_{n+1}}(m)  = \dim  R^{\La_2}_{D_{n+1}}(m)  = \sum_{\la \in \Par_n, \la \vdash m} |\ST(\la)|^2 $
\end{cor}
\begin{proof}
It is clear that $R^{\La_1}_{D_{n+1}}(m) \simeq R^{\La_2}_{D_{n+1}}(m)$ by the symmetry of the Dynkin diagram.
We shall define a simple $R^{\La_1}_{D_{n+1}}(m)$-module using shifted Young tableaux. Let $\la \vdash m$ be a strict partition. We define 
$$
\res_D(i,j) := 
\begin{cases}
1  &\text{ if $i=j$ and $i$ is odd,}   \\
2  &\text{ if $i=j$ and $i$ is even,}   \\
j-i+2 &\text{ otherwise,}
\end{cases}  
$$
and $\wt_D(\la) := \La_1 - \sum_{b\in \la} \alpha_{\res_D(b)}$. For $T \in \ST(\lambda)$, we set 
$$ 
\res_D(T) = (  \res_D(i_m,j_m), \ldots , \res_D(i_1,j_1) ),
$$
where $(i_k,j_k)\in\lambda$ is the box filled with $k$ in $T$.
For example, the residue are given as follows:
$$
\xy
(0,12)*{};(54,12)*{} **\dir{-};
(0,6)*{};(54,6)*{} **\dir{-};
(6,0)*{};(42,0)*{} **\dir{-};
(12,-6)*{};(30,-6)*{} **\dir{-};
(18,-12)*{};(24,-12)*{} **\dir{-};
(0,12)*{};(0,6)*{} **\dir{-};
(6,12)*{};(6,0)*{} **\dir{-};
(12,12)*{};(12,-6)*{} **\dir{-};
(18,12)*{};(18,-12)*{} **\dir{-};
(24,12)*{};(24,-12)*{} **\dir{-};
(30,12)*{};(30,-6)*{} **\dir{-};
(36,12)*{};(36,0)*{} **\dir{-};
(42,12)*{};(42,0)*{} **\dir{-};
(48,12)*{};(48,6)*{} **\dir{-};
(54,12)*{};(54,6)*{} **\dir{-};
(3,9)*{1}; (9,9)*{3}; (15,9)*{4}; (21,9)*{5}; (27,9)*{6}; (33,9)*{7}; (39,9)*{8};(45,9)*{9};(51,9)*{10}; 
           (9,3)*{2}; (15,3)*{3}; (21,3)*{4}; (27,3)*{5}; (33,3)*{6}; (39,3)*{7};
                      (15,-3)*{1}; (21,-3)*{3}; (27,-3)*{4};
                                   (21,-9)*{2};
\endxy
$$
Then one can show that there exists $b \in B(\La_1)$ with $\wt(b) = \wt_D(\la)$.
We set 
$$ 
S(b) := \bigoplus_{T \in \ST(\la)} \bR T,
$$
and define an $R_{D_{n+1}}$-module structure on $S(b)$ by  
$$
e(\nu) T = \delta_{ \nu, \res_D(T)} T , \quad x_i e(\nu) T = 0, \quad  
\tau_j e(\nu) T =
\begin{cases}
s_j T  &\text{ if } \nu = \mu \text{ and } s_j T \in \ST(\la),  \\
0 &\text{ otherwise,}
\end{cases}  
$$
for $T \in \ST(\la)$ and admissible $i,j, \nu $. 
Then it is straightforward to check that $S(b)$ is a homogeneous $R_{D_{n+1}}$-module and it is isomorphic to $\Sp(b)$ given in Section \ref{Sec: minuscule ADE}.

Therefore, for each $\la \vdash m $, $R^{\La_1}_{B_n}( \wt(\la)) $ and $ R^{\La_1}_{D_{n+1}}(\wt_D(\la))$ are simple and their simple modules have the same dimension. Thus
we conclude that $R^{\La_1}_{B_n}( \wt(\la)) \simeq R^{\La_1}_{D_{n+1}}(\wt_D(\la))$, which implies the desired result.
\end{proof}

\vskip 2em


\bibliographystyle{amsplain}


\end{document}